\documentclass[11pt]{article}
\usepackage{amsmath}
\usepackage{amssymb}
\usepackage{amsthm}
\usepackage{mathabx}
\usepackage[usenames]{color}
\usepackage{amscd}
\usepackage{dsfont}
\usepackage{indentfirst}

\usepackage[colorlinks=true,linkcolor=blue,filecolor=red,
citecolor=webgreen]{hyperref}
\definecolor{webgreen}{rgb}{0,.5,0}

\def\N{{\mathds{N}}}

\def\1{{\bf 1}}

\newtheorem{theorem}{Theorem}
\newtheorem{thm}[theorem]{Theorem}
\newtheorem{lemma}[theorem]{Lemma}
\newtheorem{cor}[theorem]{Corollary}

\newtheorem{prop}[theorem]{Proposition}

\begin{document}

\title{Coefficients of (inverse) unitary cyclotomic polynomials}
\author{G.\,Jones, P.\,I. Kester, L.\,Martirosyan, P.\,Moree, 
L.\,T\'oth, B.\,B.\,White and B.\,Zhang}
\maketitle

\begin{abstract}
The notion of block divisibility naturally leads one to introduce
unitary cyclotomic polynomials $\Phi_n^*(x)$. They can be written as certain products
of cyclotomic poynomials. 
We study the case where $n$ has two or three distinct prime factors using
numerical semigroups, respectively
Bachman's inclusion-exclusion polynomials. Given $m\ge 1$ we show that every integer
occurs as a coefficient of $\Phi^*_{mn}(x)$ for some $n\ge 1$ 
following Ji, Li and Moree \cite{JLM}. 
Here $n$ will 
typically have many different prime factors.
We also consider similar questions for the polynomials
$(x^n-1)/\Phi_n^*(x),$ the inverse unitary cyclotomic
polynomials. 
\end{abstract}

\section{Introduction}

\subsection{(Inverse) (unitary) cyclotomic polynomials}
The \emph{cyclotomic polynomials} $\Phi_n(x)$ are defined by
\begin{equation*}
\Phi_n(x)=\prod_{\substack{j=1\\ (j,n)=1}}^n \left(x- \exp(2\pi ij /n) \right).
\end{equation*}
They are monic polynomials of degree $\varphi(n)$ (with $\varphi$ 
Euler's totient function) and arise as irreducible factors on
factorizing $x^n-1$ over the rationals:
\begin{equation}
    \label{factor}
    x^n-1=\prod_{d|n}\Phi_d(x).
\end{equation}
By M\"obius inversion it follows from \eqref{factor} that
\begin{equation} \label{Phi}
\Phi_n(x)=\prod_{d\mid n} \left(x^{n/d}-1\right)^{\mu(d)}=\prod_{d\mid n} \left(x^{d}-1\right)^{\mu(n/d)},
\end{equation}
where $\mu$ denotes the M\"obius function.

A divisor $d$ of $n$ ($d,n\in \N$) is called a \emph{unitary divisor} (or block divisor) if $(d,n/d)=1$, notation $d\mid\mid n$ (this is in agreement with the standard notation $p^a \mid\mid n$ used for prime powers $p^a$).
If in \eqref{factor} one only considers block divisors $d$ of $n,$ 
the resulting factors are the \emph{unitary cyclotomic polynomials} $\Phi^*_d(x),$ that is, we have 
\begin{equation} \label{prod_Phi_star}
x^n-1 = \prod_{d\mid\mid n} \Phi^*_d(x).
\end{equation}
Just as the system of equations  
\eqref{factor} (taking $n=1,2,\ldots$) implicitly uniquely defines
the cyclotomic polynomials, so does the latter system of equations uniquely define
the unitary cyclotomic polynomials (the reader preferring an explicit definition 
is referred to \eqref{Phi_Phi_star}).
The polynomial $\Phi^*_n(x)$ 
is monic, has integer coefficients and is
of degree $\varphi^*(n),$ with 
 $\varphi^*(n)=\# \{j: 1\le j \le n, (j,n)_*=1\}$
 and 
$(j,n)_*=\max \{d: d\mid j, d\mid\mid n\}.$
If $n=\prod q_i$ is the factorization of $n$ in pairwise
coprime prime powers, then $\varphi^*(n)=\prod_i (q_i-1).$
Note that $\varphi^*(n)\ge \varphi(n)$.

The unitary equivalent of \eqref{Phi} reads
\begin{equation} \label{Phi_star}
\Phi^*_n(x)=\prod_{d\mid\mid n} \left(x^{n/d}-1\right)^{\mu^*(d)}=\prod_{d\mid\mid n} \left(x^{d}-1\right)^{\mu^*(n/d)},
\end{equation}
where $\mu^*(n)=(-1)^{\omega(n)}$ and 
$\omega(n)$ denotes the number of distinct prime
factors of $n$. Note that since 
$\sum_{d||n}\mu^*(d)=0$ for $n>1,$ we can alternatively
write, for $n>1,$
\begin{equation}
    \label{variation}
    \Phi_n^*(x)=\prod_{d||n}
    \left(1-x^d\right)^{\mu^*(n/d)}.
\end{equation}

Comparison of \eqref{factor} and 
\eqref{prod_Phi_star} shows that $\Phi^*_n(x)$ is a product of 
cyclotomic polynomials. The next 
result, proved in 
\cite{MorToth} where also many
connections with the theory of
arithmetic functions are pointed 
out, makes this precise.
\begin{theorem}[Moree and T\'oth \cite{MorToth}] 
\label{Th_pol}
For any natural number $n$ we have
\begin{equation} 
\label{Phi_Phi_star}
\Phi^*_{n}(x)= \prod_{\substack{d\mid n\\ \kappa(d)=\kappa(n)}} \Phi_d(x),
\end{equation}
where $\kappa(n)=\prod_{p|n}p$ is the square-free kernel of $n.$
\end{theorem}
\begin{cor} \label{cyclo=uni}
If $n$ is square-free, then $\Phi^*_n(x)=\Phi_n(x)$.
\end{cor}
This corollary is easily proved directly. In case $n$ is square-free, $\mu^*(n)=\mu(n)$ and we see that the products in \eqref{Phi} and
\eqref{Phi_star} are identical and therefore 
$\Phi_n(x)=\Phi_n^*(x)$.

By \eqref{factor} and \eqref{prod_Phi_star} we have, respectively,
\begin{equation}
\label{dubbel}
\Psi_n(x):=\frac{x^n-1}{\Phi_n(x)}
=\prod_{d<n,~d\mid n}\Phi_d(x)\text{, and~}\Psi^*_n(x):=\frac{x^n-1}{\Phi^*_n(x)}
=\prod_{d<n,~d\mid \mid n}\Phi^*_d(x),
\end{equation}
and thus both $\Psi_n$ and $\Psi_n^*$ are
polynomials having integer
coefficients.

The polynomials $\Psi_n$ were dubbed
\emph{inverse cyclotomic polynomials} by
Moree \cite{Pieter}, who was the first to
systematically study them. Meanwhile their
study found some application in cryptography, 
see, e.g., \cite{Dunand,HLLP}. Their coefficients
behave in various aspects very similar, but also in
various aspects quite dissimilar to the ordinary
cyclotomic coefficients.

The polynomials $\Psi_n^*$ seem not have be
systematically considered before. We will
call them \emph{inverse 
unitary cyclotomic polynomials}. 

The behavior of (inverse) cyclotomic coefficients is and was a topic of intense
study. The aim of this paper is to initiate the study of the coefficients
of $\Phi^*_n$ and $\Psi^*_n$.

\subsection{(Inverse) (unitary) cyclotomic coefficients}
\noindent We write 
\begin{equation}
    \label{coeffdef}
\Phi_n(x)=\sum_{j=0}^{\infty}a_n(j)x^j,
~\Psi_n(x)=\sum_{j=0}^{\infty}c_n(j)x^j,
~\Phi_n^*(x)
=\sum_{j=0}^{\infty}a_n^*(j)x^j,
~\Psi_n^*(x)
=\sum_{j=0}^{\infty}c_n^*(j)x^j.
\end{equation}
This notation looks perhaps strange to the reader, but implicitly 
defines the coefficients for every $j,$ which serves our purposes.

It turns out that for many $n$ the above polynomials 
are \emph{flat} (that is, they have all their coefficients in  $\{-1,0,1\}$).
The smallest $n$ for which the above four classes of polynomials are non-flat
are, respectively, $105,561,60$ and $120$, see the tables in Section \ref{sec:numerics}.
These tables perhaps also suggest that each of
the four polynomial families has every integer occurring as
a coefficient. The main result of this paper is that this is indeed
the case.

Ji, Li and 
Moree \cite[Theorem 1]{JLM} showed that given a fixed integer $m\ge 1$ we have
\begin{equation}
    \label{oud}
\{a_{mn}(j):n\ge 1,~j\ge 0\}=\{c_{mn}(j):n\ge 1,~j\ge 0\}=\mathbb Z.
\end{equation}
By a similar approach we will establish the following result.
\begin{theorem} 
\label{main}
Let $m\ge 1$ be fixed.
We have 
$$\{a^*_{mn}(j):n\ge 1,~j\ge 0\}=\{c^*_{mn}(j):n\ge 1,~j\ge 0\}=\mathbb Z.$$
\end{theorem}
The proof will show that we actually can restrict to the
case where $n$ is square-free and coprime
to $m$, cf. \eqref{square-free}.
The result of Ji, Li and Moree in case $m=1$ is due
to Suzuki \cite{Suzi}, who adapted a proof of 
Issai Schur (see, e.g., Emma Lehmer \cite{Lehmer}) showing that every negative
even number occurs as a
cyclotomic coefficient. Theorem \ref{main} in case $m$ is a prime power is 
due to Ji and Li \cite{JL}.

\subsection{Proof of Theorem \ref{main}}
\label{subcyclo}
\noindent Inspection of the proof of Theorem 1 in \cite{JLM} shows that the authors prove more 
than they claim, namely they show that
\begin{equation}
\label{square-free}
\{a_{mn}(j):n\ge 1,~n\text{~is~square-free},~(n,m)=1,~j\ge 0\}=\mathbb Z,
\end{equation}
and the same result with $a_{mn}(j)$ replaced by $c_{mn}(j)$.
\begin{prop}
\label{allcoeff}
Let $m\ge 1$ be square-free. We have 
$$\{a_{mn}^*(j):n\ge 1,~j\ge 0\}=\{c^*_{mn}(j):n\ge 1,~j\ge 0\}=\mathbb Z.$$
\end{prop}
\begin{proof}
If $m$ is square-free, then the index $mn$ of any coefficient appearing on
the left hand side of \eqref{square-free} is square-free. By Corollary \ref{cyclo=uni} it then follows
that $a_{mn}(j)=a^*_{mn}(j)$ and so the result follows from \eqref{square-free} for
the unitary cyclotomic coefficients. Likewise it follows for the inverse unitary cyclotomic coefficients.
\end{proof}

\noindent We are now ready to prove Theorem \ref{main}. By the latter result we could restrict to
non-square-free $m.$ However, this is not necessary as our argument works for
every $m>1.$

\begin{proof}[Proof of Theorem \ref{main}]
The result for $m=1$ is true by Proposition \ref{allcoeff}, so we may assume that $m>1.$

Let $t\ge 1$ be arbitrary but fixed. We will show that $t-1$ appears as a coefficient 
of $\Phi^*_{mn}(x)$ for some $n\ge 1$
(and in addition some variations of this).

Let $\pi(x; d, a)$ denote the
number of primes $p\le x$ that satisfy $p\equiv a({\rm mod~}d),$ with
$a,d$ coprime integers.
A quantitative version
of Dirichlet's prime number theorem for arithmetic progressions states
that, asymptotically, $\pi(x;d,a)\sim x/(\varphi(d) \log x)$. This
implies that there exist an integer $n\ge 8m$ and
primes $p_1,p_2,\ldots,p_t$ such that
$$n<p_1<p_2<\cdots <p_t<\frac{15}{8}n{\rm ~and~}
p_j\equiv 1 ({\rm mod~} m), \quad j=1,2,\ldots,t.$$
Clearly $p_t<2p_1$.
 Let $q$ be any prime exceeding $2p_1$ and put
$$n_1=
\begin{cases}
p_1p_2\cdots p_tq & \text{~if~}t\text{~is~even};\cr
p_1p_2\cdots p_t & \text{~otherwise}.
\end{cases}
$$ 
Note that $m$ and $n_1$ are coprime and that $\mu^*(n_1)=-1$. Using
these observations we conclude that
\begin{eqnarray}
\label{bluebird}
\Phi^*_{m n_1}(x)&\equiv &\prod_{d\mid \mid mn_1,~d<2p_1}(1-x^d)^{\mu^*(\frac{m n_1}{d})} ~({\rm mod~}x^{2p_1})\cr
&\equiv &\prod_{d\mid \mid m}(1-x^d)^{\mu^*(\frac{m}{d})\mu^*(m_1)}\prod_{j=1}^t(1-x^{p_j})^{\mu^*(\frac{m n_1}{p_j})} ~({\rm mod~}x^{2p_1}) \cr
&\equiv & \Phi^*_{m}(x)^{\mu^*(m_1)}
\prod_{j=1}^t(1-x^{p_j})^{-\mu^*(m n_1)} ~({\rm mod~}x^{2p_1})\cr
&\equiv & \frac{1}{\Phi^*_{m}(x)}
\prod_{j=1}^t(1-x^{p_j})^{\mu^*(m)} ~({\rm mod~}x^{2p_1})\cr
&\equiv & \frac{1}{\Phi^*_{m}(x)}\Big(1-\mu^*(m)(x^{p_1}+\ldots+x^{p_t})\Big)~({\rm
mod~}x^{2p_1}).
\end{eqnarray} 
Let $$\frac{1}{\Phi^*_m(x)}=\sum_{j=0}^{\infty}u^*_m(j)x^j$$ be the 
Taylor expansion of $1/\Phi^*_m(x)$ around $x=0.$ Noting that, for $|x|<1,$
$$\frac{1}{\Phi^*_m(x)}=-\Psi^*_m(x)(1+x^m+x^{2m}+\cdots )$$
and $m>m-\varphi^*(m)={\rm deg}\Psi_m^*,$ we see that 
$u^*_m(j)$ is an integer that only depends on the congruence
class of $j$ modulo $m.$
Thus, in particular, 
if $k\ge p_j$ we have $u^*_{m}(k-p_j)=u^*_{m}(k-1)$ 
since by assumption $p_j\equiv 1 ({\rm mod~} m).$ Using this and 
(\ref{bluebird}) we infer that, for $p_t\le k<2p_1$,
\begin{equation}
\label{boehoe}
a^*_{m n_1}(k)=u^*_{m}(k)-\mu^*(m)\sum_{j=1}^tu^*_{m}(k-p_j)=u^*_{m}(k)-\mu^*(m)tu^*_{m}(k-1).
\end{equation}
\indent  We
consider two cases depending on whether $\mu^*(m)=1$ or $\mu^*(m)=-1$.\\
{\bf Case 1}. $\mu^*(m)=1$. In this case $m$ has at least two
block divisors $>1$ (since by assumption $m>1$). Let $1<q_1<q_2$ be
the smallest, respectively second smallest block divisor $>1$ of $m.$
Note that both $q_1$ and $q_2$ are prime powers and so $\mu^*(q_i)=-1.$
Using \eqref{variation} we see that
\begin{eqnarray}
\label{eagle}
\frac{1}{\Phi^*_m(x)} & \equiv & \frac{(1-x^{q_1})(1-x^{q_2})}{1-x} ~({\rm
mod~}x^{q_2+2})\cr
& \equiv & 1+x+x^2+\cdots+x^{q_1-1}-x^{q_2}-x^{q_2+1} ~({\rm mod~}x^{q_2+2}).
\end{eqnarray}
Thus $u^*_m(k)=1$ if $k\equiv \beta ({\rm mod~}m)$ with $\beta\in \{0,1\}$
and $u^*_m(k)=-1$ if $k\equiv \beta ({\rm mod~}m)$ with $\beta\in \{q_2,q_2+1\}$.
This in combination with (\ref{boehoe}) shows that $a^*_{m n_1}(p_t)=1-t$.
Since $n\ge 8m\ge 8q_2$ we have $p_t+q_2<15n/8+n/8=2n<2p_1,$ and hence
we may apply \eqref{boehoe} with $k=p_t+q_2$ giving rise
to $a^*_{m n_1}(p_t+q_2)=t-1$. Since $\{1-t,t-1\ | \ t\ge 1\}=\mathbb Z$ the result
follows in this case.\\
{\bf Case 2}. $\mu^*(m)=-1$. Here we notice that
$$\frac{1}{\Phi^*_m(x)}\equiv 
\begin{cases}
1-x+x^2 ~({\rm mod~}x^3) & \text{~if~}m\equiv 2({\rm mod~}4);\cr
1-x ~({\rm mod~}x^3) & \text{~otherwise}.
\end{cases}
$$
Using this we find that $a^*_{m n_1}(p_t)=-1+t$. 
Furthermore, $a^*_{m n_1}(p_t+1)=1-t$ in case $m\equiv 2({\rm mod~}4)$
and $a^*_{m n_1}(p_t+1)=-t$ otherwise. Since
$\{-1+t,-t\ | \ t\ge 1\}=\mathbb Z$ and $\{-1+t,1-t\ | \ t\ge 1\}=\mathbb Z$,
it follows that also $\{a^*_{mn}(j):n\ge 1,~j\ge 0\}=\mathbb Z$ in this case.

It remains to show that  $\{c^*_{mn}(j):n\ge 1,~j\ge 0\}=\mathbb Z$. 
By Proposition \ref{allcoeff} we may assume that $m>1.$ Let $q$ be any prime exceeding $2p_1$ and put
$$n_2=
\begin{cases}
p_1p_2\cdots p_t & \text{~if~}t\text{~is~even};\cr
p_1p_2\cdots p_tq & \text{~otherwise}.
\end{cases}
$$
Using that $\mu^*(n_2)=-\mu^*(n_1),$ we see that
$$\Psi^*_{m n_2}(x)=
\frac{x^{m n_2}-1}{\Phi^*_{m n_2}(x)}\equiv -\frac{1}{\Phi^*_{m n_2}(x)}\equiv 
-\Phi^*_{m n_1}(x) ~({\rm
mod~}°x^{2p_1}),$$
and hence
$c^*_{m n_2}(k)=-a^*_{m n_1}(k)$ for $k<2p_1$. 
The proof is now completed by
reasoning 
as in the case of unitary cyclotomic coefficients using formula
\eqref{boehoe} (which is valid for 
$p_t\le k<2p_1$).
\end{proof}
\section{Connection with numerical semigroups}
Let $a_1,\ldots,a_m$ be positive integers, and let $S(a_1,\ldots,a_m)$ be the
set of all non-negative integer linear combinations of $a_1,\ldots,a_m$, that is,
$$S(a_1,\ldots,a_m)=\{x_1a_1+\ldots +x_m a_m~|~x_i\in \mathbb Z_{\ge 0}\}.$$
Then
$S$ is a {\it semigroup} (i.e., it is closed under addition). A semigroup $S$ is
said to be {\it numerical} if its complement $\mathbb Z_{\ge 0}\backslash S$ is finite. The numbers in this set are called \emph{gaps}.
It is easy to prove that $S(a_1,\ldots,a_m)$ is numerical if and ony if $a_1,\ldots,a_m$ 
are relatively prime. If $S$ is numerical,
the maximum gap is called the {\it Frobenius number} of $S$ and denoted
by $F(S)$. 
The {\it Hilbert series} of the numerical semigroup $S$ is the formal
power series $H_S(x)=\sum_{s\in S}x^s\in \mathbb Z[[x]].$ 
It is practical to multiply this by $1-x$ as we then obtain a {\it polynomial}, called
the {\it semigroup polynomial}:
\begin{equation}
\label{semigppoly}
P_S(x)=(1-x)H_S(x)=x^{F(S)+1}+(1-x)\sum_{0\le s\le F(S)\atop s\in S}x^s=1+(x-1)\sum_{s\not\in S}x^s.
\end{equation}
It is easy to see that the non-zero coefficients of $P_S$ alternate between 1 and $-1.$
{}From  $P_S$ one immediately reads off the
Frobenius number:
\begin{equation}
\label{frob}
F(S)={\rm deg}(P_S(x))-1.
\end{equation}
The following result is well-known, 
see, e.g., Bardomero and Beck \cite{Babeck}, Moree \cite{Mor2014} or 
Ram\'{\i}rez--Alfons\'{\i}n \cite[p. 34]{RA}. It seems to have been first proved
by Sz\'ekely and Wormald \cite{SW}. 
\begin{thm}
\label{basic}
If $a,b>1$ are coprime integers, then
$$P_{S(a,b)}(x)=(1-x)\sum_{s\in S(a,b)}x^s=\frac{(x^{ab}-1)(x-1)}{(x^a-1)(x^b-1)}.$$ 
\end{thm}
Using \eqref{frob} it follows that $F(S(a,b))=ab-a-b,$ something that was already known to Sylvester in
the 19th century. 

The next result is a consequence of
\eqref{Phi_star} and Theorem \ref{basic}.
\begin{theorem}
\label{twoprimepowers}
Let $p$ and $q$ be coprime prime powers $>1$. We have
$P_{S(p,q)}(x)=\Phi^*_{pq}(x).$ 
\end{theorem}
\begin{cor}
\label{eight}
We have
$$a^*_{pq}(k)=
\begin{cases}
1 &\text{if~}k\in S(p,q),~k-1\not\in S(p,q);\cr
-1 &\text{if~}k\not\in S(p,q),~k-1\in S(p,q);\cr
0 & otherwise.
\end{cases}
$$
\end{cor}
\begin{cor}
In case $p$ and $q$ are distinct primes, we have
$P_{S(p,q)}(x)=\Phi_{pq}(x).$
\end{cor}
The interpretation of $\Phi_{pq}(x)$ as a semigroup polynomial leads to 
trivial proofs of very classical facts about these so-called binary
cyclotomic polynomials. E.g., that they are of height 1
(which was
first proved by Migotti \cite{Migotti} and 
several years later by Bang \cite{Bang}) and that the non-zero coefficients 
alternate between 1 and -1 (due to Carlitz \cite{Canumberofterms}).

The polynomial $\Psi^*_{ab}(x)$ with $a<b$ coprime
prime powers, in contrast to $\Phi^*_{ab}(x),$ is boring: 
\begin{equation}
\label{psistarflat}
\Psi^*_{ab}(x)=\frac{(x^a-1)(x^b-1)}{(x-1)}=-1-x-\cdots-x^{a-1}+x^{b}+\cdots+x^{a+b-1}.
\end{equation}

For further reading on the connection between numerical semigroups and cyclotomic polynomials
the reader is referred to Moree \cite{Mor2014}.

\section{Connection 
with inclusion-exclusion
polynomials}
Let $\rho=\{r_1,r_2,\ldots,r_s\}$ be a set of pairwise coprime natural numbers 
$>1$ and put
$$n_0=\prod_i r_i,~n_i=\frac{n_0}{r_i},~n_{ij}=\frac{n_0}{r_ir_j}~[i\ne j],\ldots ,$$
and define
\begin{equation}
\label{Qformula}
Q_{\rho}(x):=\frac{(x^{n_0}-1)\cdot \prod_{i<j}(x^{n_{ij}}-1)\cdots}
{\prod_i (x^{n_i}-1)\cdot \prod_{i<j<k}(x^{n_{ijk}}-1)\cdots}.
\end{equation}
It can be shown that $Q_{\rho}(x)$ is a polynomial with \emph{integer}
coefficients. 
This class of polynomials was introduced by 
Bachman \cite{Ba}, who 
named them \emph{inclusion-exclusion polynomials}. 
From the definition and Theorem \ref{twoprimepowers} we infer that
$P_{S(a,b)}(x)=Q_{\{a,b\}}(x).$

Let $n>1$ be an integer and 
$\prod_{i=1}^t p_i^{e_i}$ its canonical
factorization. Comparison of
\eqref{Phi_star} and 
\eqref{Qformula} then shows that 
\begin{equation}
\label{gelijk}
\Phi_n^*(x)=Q_{\{p_1^{e_1},\ldots, p_t^{e_t}\}}(x).
\end{equation}
We will now derive some consequences of this identity in the ternary case $t=3.$
One of the tools that can be used here is a fundamental lemma of Kaplan 
\cite{Kaplan} relating the case $t=3$ to the case $t=2$  (he 
formulated it for cyclotomic polynomials). 

Given a polynomial $f,$ we let ${\cal C}(f)$ denote
the set of all coefficients of
$f$ and $H(f)$ the maximum
element (in absolute value) in
 ${\cal C}(f)$. Combination of \eqref{gelijk} and 
 \cite[Theorem 3]{Ba} leads to the first assertion
 below. Combination of \eqref{gelijk} and 
 \cite[Theorem]{BaMo} leads to the second assertion.
 \begin{thm}
 \label{Gennady}
Let
$p,q,r,s\ge 3$ be four pairwise coprime
prime powers. Then
${\cal C}(\Phi^*_{pqr})$ is a string of consecutive integers, and for 
$r,s>\max(p,q),$ we have

$${\cal C}(\Phi^*_{pqr}) =
\begin{cases}
{\cal C}(\Phi^*_{pqs}) & \text{~if~}r\equiv s\,({\rm mod~}pq);\cr
-{\cal C}(\Phi^*_{pqs}) & 
\text{~if~}r\equiv -s\,({\rm mod~}pq).
\end{cases}
$$
If $r\equiv \pm s\,({\rm mod~}pq)$
and $r>\max(p,q)>s\ge 3,$ then
\begin{equation}
\label{nietgelijk}
H(\Phi^*_{pqs})\le H(\Phi^*_{pqr})\le H(\Phi^*_{pqs})+1.
\end{equation}
\end{thm}
The following is a consequence of Kaplan's work, cf. \cite[(4)]{BaMo}.
\begin{cor}
Let  $p^a$  and  $q^b$  be two fixed coprime prime powers and let  $r$  be a third prime. 
Then  $\Phi^*_{ p^a q^b r^c }$  is flat for every positive exponent  $c$  with  
$r^c \equiv \pm 1 \,({\rm mod~}p^a q^b).$
\end{cor}
See subsection \ref{numsub} some numerical material demonstrating Theorem \ref{Gennady}.

\subsection{Ternary inverse unitary cyclotomic polynomials}
It seems that most (but not all!) of the work on $\Psi_{pqr}$ can be easily adapted to the 
inverse unitary setting. We merely give one example here.
\begin{theorem}
\label{upper}
Let $p<q<r$ pairwise coprime prime powers. We have
$$H(\Psi^*_{pqr})\le\Big[\frac{(p-1)(q-1)}{r}\Big]+1\le p-1.$$
\end{theorem}
\begin{proof} 
Using \eqref{Phi} and \eqref{Phi_star} we see that
$\Psi^*_{pqr}(x)=\Phi^*_{pq}(x)\Psi^*_{pq}(x^r),$
and so
\begin{equation}
\label{startie}
c^*_{pqr}(k)=\sum_{j=0}^{[\frac{k}{r}]}a^*_{pq}(k-jr)c^*_{pq}(j).
\end{equation}
The number of $j$ for which $0\le k-jr\le \varphi^*(pq),$ and so
$a^*_{pq}(k-jr)$ is potentially non-zero, is at most
$$\Big[\frac{\varphi^*(pq)}{r}\Big]+1=\Big[\frac{(p-1)(q-1)}{r}\Big]+1\le p-2+1=p-1.$$
Since $|a^*_{pq}(k-jr)|\le 1$ by 
Corollary \ref{eight} and 
$|c^*_{pq}(j)|\le 1$ by \eqref{psistarflat}, the proof is concluded.
\end{proof}

\section{Some numerical data}
\label{sec:numerics}
Let $n\ge 1$ and 
let $f_n\in \{\Phi_n,\Psi_n,\Phi^*_n,\Psi^*_n\},$ with its coefficients
denoted as in \eqref{coeffdef}.
For a given integer
$m\ge 2,$ we list the smallest $n$ such that 
$H(f_n)=m,$ 
with $H(f)$ the maximum
coefficient (in absolute value) of $f$.
In
addition we list the
degree of $f_{n},$ the smallest $k$ such
that $|f_{n}(k)|=m,$ and the value of  $f_{n}(k)$.\\

\centerline{{\bf Table 1:} {\tt ($\Phi_n$) Minimal $n$ and $k$ with $|a_n(k)|=m$}}
\begin{center}
\begin{tabular}{|c|l|c|c|c|c|}
\hline
$m$    & $n$ & {\rm deg}$(\Phi_{n})$ & $k$  & $a_{n}(k)$\\
\hline
$2$    & $105=3\cdot5\cdot 7$ & $48$ & $7$ & $-2$\\
\hline
$3$ & $385=5\cdot 7\cdot 11$ & $240$ & $119$ & $-3$\\
\hline
$4$ & $1365=3\cdot 5\cdot7\cdot 13$ & $576$ & $196$ & $-4$\\
\hline
$5$  & $1785=3\cdot 5\cdot7\cdot 17$ & $768$ & $137$ & $+5$\\
\hline
$6$  & $2805=3\cdot 5\cdot 11\cdot 17$ & $1280$ & $573$ & $-6$\\
\hline
$7$  & $3135=3\cdot 5\cdot 11\cdot 19$ & $1440$ & $616$ & $+7$\\
\hline
$8$  & $6545=5\cdot 7\cdot 11\cdot 17$ & $3840$ & $1528$ & $-8$\\
\hline
$9$  & $6545=5\cdot 7\cdot 11\cdot 17$ & $3840$ & $1914$ & $+9$\\
\hline
$10$ & $10465=5\cdot 7\cdot 13\cdot 23$ & $6336$ & $1196$ & $-10$\\
\hline
$11$  & $10465=5\cdot 7\cdot 13\cdot 23$ & $6336$  & $1916$ & $-11$\\
\hline
\end{tabular}
\end{center}
For $m=10,\ldots,14$ it turns out that $n=10465.$
\vfil\eject
\centerline{{\bf Table 2:} {\tt ($\Psi_n$) Minimal $n$ and $k$ with $|c_n(k)|=m$}}
\begin{center}
\begin{tabular}{|c|c|c|c|c|c|}
\hline
$m$    & $n$ & {\rm deg}$(\Psi_{n})$ & $k$  & $c_{n}(k)$\\
\hline
$2$    & $561=3\cdot11\cdot 17$ & $241$ & $17$ & $-2$\\
\hline
$3$ & $1155=3\cdot 5\cdot 7\cdot 11$ & $675$ & $33$ & $-3$\\
\hline
$4$ & $2145=3\cdot 5\cdot 11\cdot 13$ & $1185$ & $44$ & $+4$\\
\hline
$5$  & $3795=3\cdot 5\cdot 11\cdot 23$ & $2035$ & $132$ & $-5$\\
\hline
$6$  & $5005=5\cdot 7\cdot 11\cdot 13$ & $2125$ & $201$ & $-6$\\
\hline
$7$  & $5005=5\cdot 7\cdot 11\cdot 13$ & $2125$ & $310$ & $-7$\\
\hline
$8$  & $8645=5\cdot 7\cdot 13\cdot 19$ & $3461$ & $227$ & $-8$\\
\hline
$9$  & $8645=5\cdot 7\cdot 13\cdot 19$ & $3461$ & $240$ & $+9$\\
\hline
$10$  & $11305=5\cdot 7\cdot 17\cdot 19$ & $4393$ & $240$ & $-10$\\
\hline
$11$  & $11305=5\cdot 7\cdot 17\cdot 19$ & $4393$ & $306$ & $+11$\\
\hline
\end{tabular}
\end{center}
For $m=10,\ldots,21$ it turns out that $n=11305$.\\

\centerline{{\bf Table 3:} {\tt $(\Phi^*_n)$ Minimal $n$ and $k$ with $|a^*_n(k)|=m$}}
\begin{center}
\begin{tabular}{|c|l|c|c|c|c|}
\hline
$m$    & $n$ & {\rm deg}$(\Phi_{n}^*)$ & $k$  & $a^*_{n}(k)$\\
\hline
$2$    & $60=2^2\cdot3\cdot 5$ & $24$ & $5$ & $-2$\\
\hline
$3$ & $385=5\cdot 7\cdot 11$ & $240$ & $119$ & $-3$\\
\hline
$4$ & $780=2^2\cdot3\cdot 5\cdot 13$ & $288$ & $78$ & $-4$\\
\hline
$5$  & $1320=2^3\cdot 3\cdot 5\cdot 11$ & $560$ & $107$ & $-5$\\
\hline
$6$  & $1320=2^3\cdot 3\cdot 5\cdot 11$ & $560$ & $111$ & $+6$\\
\hline
$7$  & $1320=2^3\cdot 3\cdot 5\cdot 11$ & $560$ & $210$ & $-7$\\
\hline
$8$  & $1320=2^3\cdot 3\cdot 5\cdot 11$ & $560$ & $213$ & $-8$\\
\hline
$9$  & $3640=2^3\cdot 5\cdot 7\cdot 13$ & $2016$ & $626$ & $-9$\\
\hline
$10$ & $3640=2^3\cdot 5\cdot 7\cdot 13$ & $2016$ & $648$ & $+10$\\
\hline
$11$  & $3640=2^3\cdot 5\cdot 7\cdot 13$ & $2016$ & $748$ & $+11$\\
\hline
$12$  & $3640=2^3\cdot 5\cdot 7\cdot 13$ & $2016$ & $761$ & $+12$\\
\hline
$13$  & $4620=2^2\cdot 3\cdot 5\cdot 7\cdot 11$ & $1440$ & $386$ & $-13$\\
\hline
$14$  & $4620=2^2\cdot 3\cdot 5\cdot 7\cdot 11$ & $1440$ & $419$ & $-14$\\
\hline
$15$  & $4620=2^2\cdot 3\cdot 5\cdot 7\cdot 11$ & $1440$ & $425$ & $+15$\\
\hline
$16$  & $4620=2^2\cdot 3\cdot 5\cdot 7\cdot 11$ & $1440$ & $474$ & $-16$\\
\hline
$17$  & $4620=2^2\cdot 3\cdot 5\cdot 7\cdot 11$ & $1440$ & $497$ & $-17$\\
\hline
$18$  & $4620=2^2\cdot 3\cdot 5\cdot 7\cdot 11$ & $1440$ & $475$ & $-18$\\
\hline
$19$  & $4620=2^2\cdot 3\cdot 5\cdot 7\cdot 11$ & $1440$ & $558$ & $+19$\\
\hline
\end{tabular}
\end{center}
For $m=20,\ldots,41$ it turns out that $n=9240$.\\

\vfil\eject
\centerline{{\bf Table 4:} {\tt ($\Psi^*_n$) Minimal $n$ and $k$ with $|c^*_n(k)|=m$}}
\begin{center}
\begin{tabular}{|c|l|c|c|c|c|}
\hline
$m$    & $n$ & {\rm deg}$(\Psi_{n}^*)$ & $k$  & $c^*_{n}(k)$\\
\hline
$2$    & $120=2^3\cdot3\cdot 5$ & $64$ & $8$ & $-2$\\
\hline
$3$ & $420=2^2\cdot 3\cdot 5\cdot 7$ & $276$ & $12$ & $-3$\\
\hline
$4$ & $1008=2^4\cdot3^2\cdot 7$ & $288$ & $48$ & $-4$\\
\hline
$5$  & $1820=2^2\cdot 5\cdot 7\cdot 13$ & $956$ & $475$ & $+5$\\
\hline
$6$  & $3080=2^3\cdot 5\cdot 7\cdot 11$ & $1400$ & $66$ & $+6$\\
\hline
$7$  & $3080=2^3\cdot 5\cdot 7\cdot 11$ & $1400$ & $103$ & $+7$\\
\hline
$8$  & $3080=2^3\cdot 5\cdot 7\cdot 11$ & $1400$ & $114$ & $-8$\\
\hline
$9$  & $3080=2^3\cdot 5\cdot 7\cdot 11$ & $1400$ & $111$ & $-9$\\
\hline
$10$  & $3080=2^3\cdot 5\cdot 7\cdot 11$ & $1400$ & $112$ & $-10$\\
\hline
$11$  & $3080=2^3\cdot 5\cdot 7\cdot 11$ & $1400$ & $121$ & $+11$\\
\hline
$12$  & $3080=2^3\cdot 5\cdot 7\cdot 11$ & $1400$ & $122$ & $+12$\\
\hline
$13$  & $3080=2^3\cdot 5\cdot 7\cdot 11$ & $1400$ & $177$ & $+13$\\
\hline
$14$  & $9240=2^3\cdot 3\cdot5\cdot 7\cdot 11$ & $5880$& $261$ & $-14$\\
\hline
$15$  & $8580=2^2\cdot 3\cdot 5\cdot 11\cdot 13$ & $5700$ & $705$ & $-15$\\
\hline
$16$  & $9240=2^3\cdot 3\cdot5\cdot 7\cdot 11$ & $5880$ & $253$ & $-16$\\
\hline
$17$  & $9240=2^3\cdot 3\cdot5\cdot 7\cdot 11$ & $5880$ & $325$ & $+17$\\
\hline
$18$  & $9240=2^3\cdot 3\cdot5\cdot 7\cdot 11$ & $5880$ & $341$ & $+18$\\
\hline
$19$  & $9240=2^3\cdot 3\cdot5\cdot 7\cdot 11$ & $5880$ & $450$ & $+19$\\
\hline
\end{tabular}
\end{center}
For m=$16,\ldots,21$ it turns out that $n=9240$.\\

The tables suggest that the (unitary) cyclotomic polynomials are generically
of the same flatness as their inverses. However, generically an (inverse) cyclotomic
polynomial seems to be rather flatter than its unitary equivalent.

\subsection{Numerical material related to Theorem \ref{Gennady}}
\label{numsub}
If $n$ has three or less block divisors that are prime powers, then 
${\cal C}(\Phi^*_{n})$ consists of consecutive integers:
we have, e.g.\,, 
${\cal C}(\Phi^*_{8\cdot 11\cdot 13}) = \{-4, 3\}$ and
${\cal C}(\Phi^*_{27\cdot 29\cdot 31}) = \{-8,\ldots,8\}.$
If $n$ has four or more block divisors that are prime powers, this is not always true: we have, e.g., 
${\cal C}(\Phi^*_{2^4\cdot 3^2\cdot 5^2\cdot 7}) = \{-49,\ldots,44\}\backslash \{-48,-47,-45,-43,40,42,43\}.$

In practice both the upper and lower bound in \eqref{nietgelijk} are
often assumed, here we give just two examples.
\begin{itemize}
\item Let $p=2^2$, $q=5$, $s=3$, $r=23$. Then $r\equiv s\,({\rm mod~}pq)$ and 
$r>\max(p,q)>s=3$. We have $H(\Phi_{pqs}^*)=H(\Phi_{pqr}^*)=2.$
\item  Let $p=3^2$, $q=7$, $s=5$, $r=131$. Then 
$r\equiv s\,({\rm mod~}pq)$ and $r>\max(p,q)>s>3$. We have 
$H(\Phi_{pqr}^*) =H(\Phi_{pqs}^*)+1=3.$
\end{itemize}
We do not know of any simple criteria that can be used to determine which 
of the two bounds must hold.\\

\noindent {\tt Acknowledgment}. This paper is
partly based on the outcome of a student project 
by Greyson Jones, Philip Kester and Brenden White 
that they carried out under the direction of 
Lilit Martirosyan at the University of North
Carolina in Wilmington. The authors thank Gennady 
Bachman, Alexandru Ciolan and Andr\'es Herrera-Poyatos for helpful
feedback on earlier verions.

\vskip4mm
\medskip\noindent Greyson Jones, Philip Isaac Kester, 
Brenden Blake White\par\noindent
{\footnotesize e-mails: {\tt rgj5866@uncw.edu, pk7312@uncw.edu, bbw7810@uncw.edu}}
\vskip3mm

\medskip\noindent Lilit Martirosyan \par\noindent
{\footnotesize University of North Carolina, Wilmington\\
Department of Mathematics and Statistics\\ 
601 South College Road\\ Wilmington
NC 28403-5970, USA.\hfil\break
 e-mail: {\tt martirosyanl@uncw.edu}}
\vskip3mm

\noindent
Pieter Moree \\
{\footnotesize Max-Planck-Institut f\"ur Mathematik \\
Vivatsgasse 7, 53111 Bonn, Germany  \\
E-mail: {\tt moree@mpim-bonn.mpg.de}}
\vskip3mm

\noindent
L\'aszl\'o T\'oth  \\
{\footnotesize Department of Mathematics \\
University of P\'ecs \\
Ifj\'us\'ag \'utja 6, 7624 P\'ecs, Hungary \\
E-mail: {\tt ltoth@gamma.ttk.pte.hu}}
\vskip3mm

\noindent 
Bin Zhang \\
{\footnotesize School of Mathematical Sciences\\ 
Qufu Normal University\\ 
Qufu 273165, P. R. China\\
E-mail: {\tt zhangbin100902025@163.com}}

\end{document}